\documentclass[a4paper]{amsart}
\usepackage[english]{babel}
\usepackage{amsmath,amsthm,amssymb,amsfonts}

\let\Oldepsilon\epsilon
\let\Oldvarepsilon\varepsilon
  \renewcommand{\varepsilon}{\Oldepsilon}
  \renewcommand{\epsilon}{\Oldvarepsilon}

\DeclareMathOperator{\Kern}{\mathcal{K}}
\DeclareMathOperator{\supp}{\mathrm{supp}}
\DeclareMathOperator{\dom}{\mathrm{dom}}

\newcommand{\N}{\mathbb{N}}
\newcommand{\R}{\mathbb{R}}
\newcommand{\C}{\mathbb{C}}

\newcommand{\Space}{\mathrm{X}}
\newcommand{\Heis}{\mathrm{H}}

\newcommand{\dist}{\varrho}
\newcommand{\dil}{\delta}

\newcommand{\tc}{\,:\,}

\newcommand{\leftopenint}{\left]}
\newcommand{\rightopenint}{\right[}
\newcommand{\leftclosedint}{\left[}
\newcommand{\rightclosedint}{\right]}

\DeclareMathOperator*{\esssup}{\mathrm{ess\,sup}}

\newtheorem{theorem}{Theorem}
\newtheorem{proposition}[theorem]{Proposition}
\newtheorem{coro}[theorem]{Corollary}
\newtheorem{lemma}[theorem]{Lemma}

\newcommand{\vecL}{\mathbf{L}}
\newcommand{\vecT}{\mathbf{T}}
\newcommand{\vecP}{\mathbf{P}}

\newcommand{\vecONE}{\tilde{1}}

\newcommand{\jone}{{j}}
\newcommand{\jtwo}{{k}}

\begin{document}

\title[Weighted Plancherel estimates for the Grushin operators]
{Weighted Plancherel estimates and sharp spectral multipliers for the Grushin operators}

\author{Alessio Martini}
\address{Alessio Martini \\ Mathematisches Seminar \\ Christian-Albrechts-Universit\"at zu Kiel \\ Ludewig-Meyn-Str.\ 4 \\ D-24098 Kiel \\ Germany}
\email{martini@math.uni-kiel.de}

\author{Adam Sikora}
\address{Adam Sikora \\ Department of Mathematics \\ Macquarie University \\ NSW 2109 \\ Australia}
\email{adam.sikora@mq.edu.au}

\thanks{This research was supported by Australian Research Council Discovery grants DP110102488 (A.M., A.S.).
The first-named author thanks the Australian Research Council and the Alexander von Humboldt Foundation for support when this work was begun and when it was finished, and the University of New South Wales and the Christian-Albrechts-Universit\"at zu Kiel for their hospitality.}

\begin{abstract}
We study the Grushin operators acting on $\R^{d_1}_{x'}\times \R^{d_2}_{x''}$ and 
defined by the formula
\[
L=-\sum_{\jone=1}^{d_1}\partial_{x'_\jone}^2 - \left(\sum_{\jone=1}^{d_1}|x'_\jone|^2\right) 
\sum_{\jtwo=1}^{d_2}\partial_{x''_\jtwo}^2.
\]
We obtain weighted Plancherel estimates for the considered operators. 
As a consequence we prove $L^p$ spectral multiplier results and Bochner-Riesz summability
for the Grushin operators. These multiplier results are sharp if $d_1 \ge d_2$. 
We discuss also an interesting  phenomenon for weighted Plancherel estimates 
for $d_1 <d_2$. The described spectral multiplier theorem is the analogue of the result
for the sublaplacian on the Heisenberg group obtained by D. M\"uller and E.M. Stein and by
W.~Hebisch.
\end{abstract}

\maketitle

\section{Introduction}

Let $(\Space,\mu)$ be a measure space and $L$ be a (possibly unbounded) self-adjoint operator on $L^2(\Space)$. If $E$ denotes the spectral resolution of $L$ on $\R$, then a functional calculus for $L$ can be defined via spectral integration and, for every Borel function $F : \R \to \C$, the operator
\[F(L) = \int_\R F(\lambda) \,dE(\lambda)\]
is bounded on $L^2(\Space)$ if and only if the ``spectral multiplier'' $F$ is an ($E$-essentially) bounded function. Characterizing, or at least giving (non-trivial) sufficient conditions for the $L^p$-boundedness of the operator $F(L)$, for some $p \neq 2$, in terms of properties of the multiplier $F$, is a much more complicated issue, and a huge amount of literature is devoted to instances of this problem (we refer the reader to \cite{cowling_spectral_2001, duong_singular_1999, duong_plancherel-type_2002, hebisch_multiplier_1993, hebisch_functional_1995, martini_multipliers_2010, mauceri_vectorvalued_1990, mller_spectral_1994, sikora_imaginary_2001} for a detailed discussion of the relevant literature).

Here we are interested in the case where $\Space = \R^{d_1} \times \R^{d_2}$, with Lebesgue measure, and $L$ is the Grushin operator, that is,
\[L = -\Delta_{x'} - |x'|^2 \Delta_{x''},\]
where $x',x''$ denote the two components of a point $x \in \R^{d_1} \times \R^{d_2}$, while $\Delta_{x'},\Delta_{x''}$ are the corresponding partial Laplacians, and $|x'|$ is the Euclidean norm of $x'$.

Let $W_q^s(\R)$ denote the $L^q$ Sobolev space on $\R$ of (fractional) order $s$, and define a ``local Sobolev norm'' by the formula
\[
\|F\|_{MW_q^s} = \sup_{t>0} \|\eta \, F_{(t)} \|_{W_q^s},
\]
where $F_{(t)}(\lambda) =F(t \lambda)$, and $\eta \in C^\infty_c(\leftopenint 0,\infty\rightopenint)$ is a not identically zero auxiliary function; note that different choices of $\eta$ give rise to equivalent local norms. Next set $D = \max\{d_1+d_2,2d_2\}$. Then our main result reads as follows.

\begin{theorem}\label{thm:maintheorem}
Suppose that a function $F : \R \to \C$ satisfies
\begin{equation}\label{eq:mhcond}
\|F\|_{MW_2^s}  < \infty
\end{equation}
for some $s > D/2$. Then the operator $F(L)$ is of weak type $(1,1)$ and bounded on $L^p(\Space)$ for all $p \in \leftopenint1,\infty\rightopenint$. In addition,
\begin{equation}\label{eq:normbounds}
\|F(L)\|_{L^1 \to L^{1,w}} \leq C \|F\|_{MW_2^s}, \qquad \|F(L)\|_{L^p \to L^{p}} \leq C_p \|F\|_{MW_2^s}
\end{equation}
for all $r \geq 0$.
\end{theorem}

When $d_1=d_2=1$, this result proves the conjecture stated on page 5 of \cite{meyer_estimates_2006}, and in fact we obtain a far-going generalization of that statement.

Note that, in the case $d_1 \geq d_2$, the lower bound for the order of differentiability $s$ of the multiplier $F$ required in Theorem~\ref{thm:maintheorem} is $(d_1+d_2)/2$, that is, half the topological dimension of $\Space$; since the Grushin operator $L$ is elliptic in the region where $x'\neq 0$, the argument in \cite{sikora_imaginary_2001} can be adapted to show that our result is sharp (see Section~\ref{section:sharp} below for more details). In the case $d_2 > d_1$, instead, a gap of $(d_2-d_1)/2$ remains between half the topological dimension and the threshold on $s$ in Theorem~\ref{thm:maintheorem}.

If one disregarded the constraint on $s$, then the above result would follow from a general theorem \cite{hebisch_functional_1995,duong_plancherel-type_2002} proved in the context of a doubling metric-measure space $(\Space,\dist,\mu)$, with an operator $L$ satisfying Gaussian-type heat kernel bounds expressed in terms of the distance $\dist$. In this general context, the mentioned weak type and $L^p$-boundedness of $F(L)$ hold whenever $\|F\|_{MW_\infty^s} < \infty$ for some $s > Q/2$, where $Q$ denotes the ``homogeneous dimension'' of the metric-measure space.

If $\Space$ is $\R^{d_1} \times \R^{d_2}$ with Lebesgue measure and $L$ is the Grushin operator, a ``control distance'' $\dist$ associated to $L$ can be introduced, for which $Q = d_1 + 2d_2$ and $L$ satisfies Gaussian-type bounds \cite{robinson_analysis_2008}, hence the general theorem applies in this case. 
Our Theorem~\ref{thm:maintheorem} gives a better result, with respect to both the order of differentiability required on the multiplier (since $D < Q$) and the type of Sobolev norm used ($L^2$ instead of $L^\infty$). 

Our approach allows us to consider also the Bochner-Riesz means associated to the Grushin operator. Since these correspond to compactly supported multipliers, we can obtain a better result than the one given by direct application of Theorem~\ref{thm:maintheorem}.
\begin{theorem}\label{thm:maintheoremriesz}
Suppose that $\kappa > (D-1)/2$ and $p \in \leftclosedint 1,\infty\rightclosedint$. Then the Bochner-Riesz means $(1-t L)_+^\kappa$ are bounded on $L^p(\Space)$ uniformly in $t \in \leftclosedint 0,\infty\rightopenint$.
\end{theorem}

As in many other works on the subject, the proof of our results is based on the analysis of the integral kernel $\Kern_{F(L)} : \Space \times \Space \to \C$ of the operator $F(L)$, defined by the identity
\begin{equation}\label{eq:integralkernel}
F(L) f(x) = \int_\Space \Kern_{F(L)}(x,y) \, f(y) \,dy.
\end{equation}
To be precise, if $F$ is bounded and compactly supported, then there exists a Borel function $\Kern_{F(L)}$ such that \eqref{eq:integralkernel} holds for all $f \in L^2(\Space)$ and for almost all $x \in \Space$ (cf.\ \cite[Lemma 2.2]{duong_plancherel-type_2002}). However, a multiplier $F$ satisfying \eqref{eq:mhcond} need not be compactly supported, and the integral kernel $\Kern_{F(L)}$ in general exists only as a distribution; nevertheless the Calder\'on-Zygmund theory of singular integral operators allows one to derive the weak type $(1,1)$ of $F(L)$ from suitable estimates on the integral kernels corresponding to the compactly supported pieces in a dyadic decomposition of $F$.

As highlighted in \cite{duong_plancherel-type_2002}, a crucial step in this approach is a ``Plancherel estimate'', which in its basic form is the inequality
\begin{equation}\label{eq:basicplancherel}
\esssup_{y \in \Space} |B(y,R^{-1})|^{1/2} \, \|\Kern_{F( L)}(\cdot,y)\|_{L^2(\Space)} 
\leq C \|F_{(R^2)}\|_{L^\infty(\R)},
\end{equation}
for all $R > 0$ and all $F : \R \to \C$ supported in the interval $\leftclosedint R^2,4R^2 \rightclosedint$; here $|B(x,r)|$ denotes the Lebesgue measure of the $\dist$-ball of center $x \in \Space$ and radius $r$. 
Such an estimate holds, \emph{mutatis mutandis}, for any operator $L$ satisfying Gaussian heat kernel bounds, but usually it does not lead to optimal spectral multiplier results.
In the present paper we obtain for the Grushin operator $L$ an improvement of \eqref{eq:basicplancherel}, that is, a ``weighted Plancherel estimate'' of the form
\begin{equation}\label{eq:weightedplancherel}
\esssup_{y \in \Space} |B(y,R^{-1})|^{1/2} \, \|(1+w_R(\cdot,y))^\gamma \, \Kern_{F( L)}(\cdot,y)\|_{L^2(\Space)} \leq C_\gamma \|F_{(R^2)}\|_{L^2(\R)},
\end{equation}
where $\gamma \in \leftclosedint 0, d_2/2\rightopenint$ and
\begin{equation}\label{eq:weight}
w_R(x,y) = \min\{R,|y'|^{-1}\} |x'|.
\end{equation}
The improvement of the Plancherel estimate yields, via the interpolation technique of \cite{mauceri_vectorvalued_1990}, a sharp multiplier theorem, at least for $d_1 \leq d_2$. In the case $d_1 > d_2$, an interesting phenomenon occurs: although \eqref{eq:weightedplancherel} holds for all $\gamma \in \leftclosedint 0, d_2/2\rightopenint$, we can exploit it only when $\gamma < d_1/2$; whence the gap between the threshold $D/2$ in Theorem~\ref{thm:maintheorem} and half the topological dimension.

The use of weighted Plancherel estimates in multiplier theorems is not new \cite{hebisch_multiplier_1993,mller_spectral_1994,hebisch_functional_1995,cowling_spectral_2001}, and in particular they have been employed to obtain sharp results for homogeneous sublaplacians on Heisenberg and related groups. In the case $d_2 = 1$, the Grushin operator $L$ corresponds, via a suitable quotient, to the homogeneous sublaplacian on the $(2d_1+1)$-dimensional Heisenberg group. For $d_2 > 1$, an analogous correspondence holds if we replace the Heisenberg group with a Heisenberg-Reiter group $\Heis_{d_1,d_2}$ (see details below), and a multiplier theorem for the homogeneous sublaplacian on $\Heis_{d_1,d_2}$ holds \cite[Corollary 6.1]{martini_joint}, giving the same gap between the threshold and half the topological dimension that appears in Theorem~\ref{thm:maintheorem}. We remark, however, that the weighted Plancherel estimate for the Grushin operator is not an immediate consequence of the analogous estimate on the Heisenberg-Reiter group: because of the absence of the group structure, here some careful estimates are needed, exploiting known asymptotics for the Hermite functions.

In a recent work \cite{jotsaroop_riesz_2011} the Grushin operator in the case $d_2 = 1$ is considered and results analogous to our Theorems~\ref{thm:maintheorem} and \ref{thm:maintheoremriesz} are obtained. However in the present paper the requirements on the order of differentiability $s$ and on the order of Bochner-Riesz means $\kappa$ 
are essentially lower. Moreover the method used in \cite{jotsaroop_riesz_2011} apparently does not yield the weak type $(1,1)$ in the multiplier theorem, nor the $L^1$-boundedness of the Bochner-Riesz means.

\section{Notation and preliminaries}

As above, let $\Space$ be $\R^{d_1}\times \R^{d_2}$ with Lebesgue measure. In order to study the Grushin operator $L$ on $\Space$, it is convenient to introduce at the same time a family of operators which commute with $L$. 

Given a point $x = (x',x'') \in \Space$, we denote by $x'_{\jone}$ and $x''_{\jtwo}$ the $\jone$-th component of $x'$ and the $\jtwo$-th component of $x''$.
For all $\jone \in \{1,\dots,d_1\}$ and $\jtwo \in \{1,\dots,d_2\}$, let then $L_{\jone}$, $T_{\jtwo}$, $P_{\jone}$ be the differential operators on $\Space$ given by
\[L_{\jone} = (-i\partial_{x'_{\jone}})^2 + (x'_{\jone})^2 \sum_{l=1}^{d_2} (-i\partial_{x''_l})^2, \qquad T_{\jtwo} = -i\partial_{x''_{\jtwo}}, \qquad P_{\jone} = x'_{\jone}.\]
If $(\dil_r)_{r > 0}$ is the family of dilations on $\Space$ defined by
\[\dil_r(x',x'') = (rx',r^2 x''),\]
then $\|f \circ \dil_r\|_2 = r^{-Q/2} \|f\|_2$, where $Q = d_1 + 2d_2$. We also note that 
\begin{equation}\label{eq:homogeneity}
\begin{split}
P_{\jone} (f \circ \dil_r) = r^{-1} (P_{\jone} f) \circ \dil_r, &\qquad L_{\jone}(f \circ \dil_r) = r^2 (L_{\jone} f) \circ \dil_r, \\ T_{\jtwo} (f \circ \dil_r) &= r^2 (T_{\jtwo} f) \circ \dil_r.
\end{split}
\end{equation}

The Grushin operator $L$ on $\Space$ is the sum $L_1 + \dots + L_{d_1}$. 
$L$ is a second-order subelliptic differential operator with smooth coefficients. For such operators, several ways of introducing a control distance $\dist$ are available in the literature, and we refer the reader to \cite{jerison_subelliptic_1987} for a survey. In particular, $L$ belongs to the class of operators studied in \cite{robinson_analysis_2008}, where the following estimates are obtained.

\begin{proposition}
The control distance $\dist$ of the Grushin operator $L$ on $\Space$ is homogeneous with respect to the dilations $\dil_r$, that is,
\[\dist(\dil_r(x),\dil_r(y)) = r \dist(x,y)\]
for all $r > 0$ and $x,y \in \Space$, and moreover
\begin{equation}\label{eq:controldistance}
\dist(x,y) \sim |x' - y'| + \begin{cases}
\frac{|x''-y''|}{|x'| + |y'|} &\text{if $|x''-y''|^{1/2} \leq |x'| + |y'|$,}\\
|x''-y''|^{1/2} &\text{if $|x''-y''|^{1/2} \geq |x'| + |y'|$.}
\end{cases}
\end{equation}
Consequently, if $B(x,r)$ denotes the $\dist$-ball of center $x \in \Space$ and radius $r \geq 0$, then
\begin{equation}\label{eq:ballmeasure}
|B(x,r)| \sim r^{d_1+d_2} \max\{r,|x'|\}^{d_2},
\end{equation}
and in particular, for all $\lambda \geq 0$,
\begin{equation}\label{eq:doubling}
|B(x,\lambda r)| \leq C (1+\lambda)^Q |B(x,r)|.
\end{equation}
Moreover, there exist constants $b,C > 0$ such that, for all $t > 0$, the integral kernel $p_t$ of the operator $\exp(-tL)$ is a function satisfying
\begin{equation}\label{eq:gaussianbounds}
|p_t(x,y)| \leq C |B(y,t^{1/2})|^{-1} e^{-b \dist(x,y)^2/t}
\end{equation}
for all $x,y \in \Space$.
\end{proposition}
\begin{proof}
The homogeneity of $\dist$ follows immediately from its definition \cite[\S 4]{robinson_analysis_2008} and the homogeneity of $L$. For the remaining estimates, see \cite[Proposition 5.1 and Corollary 6.6]{robinson_analysis_2008}.
\end{proof}

The inequality \eqref{eq:doubling} says that $\Space$ with the distance $\dist$ and the Lebesgue measure is a doubling metric-measure space of homogeneous dimension $Q$ (cf.\ \cite[\S 2]{duong_singular_1999} or \cite[formula (2.2)]{duong_plancherel-type_2002}), whereas \eqref{eq:gaussianbounds} expresses Gaussian-type heat kernel bounds for $L$.

Several properties of $L$ and the other operators introduced above can be easily recovered by considering $\Space$ as the quotient of a suitable stratified Lie group (cf.\ \cite{beals_transversally_2004,arcozzi_grushin_2008}). Denote by $\R^{d_1 \times d_2}$ the set of $(d_1 \times d_2)$-matrices with real coefficients. Both $\R^{d_1 \times d_2}$ and $\R^{d_1} \times \R^{d_2}$ are abelian Lie groups with respect to addition. Let $\Heis_{d_1,d_2}$ be the semidirect product group $\R^{d_1 \times d_2} \ltimes (\R^{d_1} \times \R^{d_2})$, with multiplication
\[(x,y,t) \cdot (x_0, y_0, t_0) = (x+ x_0, y+ y_0, t+ t_0 - (x^T y_0 - x_0^T y)/2).\]
This is a particular instance of Heisenberg-Reiter group (see \cite{torreslopera_cohomology_1985} and references therein). If $\{\tilde X_{1,1},\dots,\tilde X_{d_1,d_2}, \tilde Y_{1}, \dots,\tilde Y_{d_1}, \tilde T_{1}, \dots, \tilde T_{d_2}\}$ is the standard basis of the Lie algebra of $\Heis_{d_1,d_2}$ (i.e., the set of the left-invariant vector fields extending the standard basis of $\R^{d_1 \times d_2} \times \R^{d_1} \times \R^{d_2}$ at the identity), then the only non-trivial Lie brackets among the elements of the basis are
\[[\tilde X_{\jone,\jtwo},\tilde Y_{\jone}] = -[\tilde Y_{\jone},\tilde X_{\jone,\jtwo}] = -\tilde T_{\jtwo} \qquad\text{for all $\jone = 1,\dots,d_1$, $\jtwo = 1,\dots,d_2$.}\]
$\Heis_{d_1,d_2}$ is a $2$-step stratified Lie group, with dilations $(\tilde\dil_r)_{r>0}$ defined by
\[\tilde\dil_r(\tilde X_{\jone,\jtwo}) = r\tilde X_{\jone,\jtwo}, \quad \tilde\dil_r(\tilde Y_{\jone}) = r\tilde Y_{\jone}, \quad \tilde\dil_r(\tilde T_{\jtwo}) = r^2 \tilde T_{\jtwo},\]
and the homogeneous sublaplacian $\tilde L$ on $\Heis_{d_1,d_2}$ is given by
\[\tilde L = -\sum_{\jone,\jtwo} \tilde X_{\jone,\jtwo}^2 - \sum_{\jone} \tilde Y_{\jone}^2.\]

Note that, when $d_2=1$, $\Heis_{d_1,d_2}$ is the $(2d_1+1)$-dimensional Heisenberg group. When $d_2 > 1$, $\Heis_{d_1,d_2}$ is not an H-type group (in the sense of Kaplan), nor a M\'etivier group. Nevertheless, in the terminology of \cite{martini_multipliers_2010,martini_joint}, $\Heis_{d_1,d_2}$ is $h$-capacious where $h= \min\{d_1,d_2\}$. In particular, the following multiplier theorem holds: the operator $F(\tilde L)$ is of weak type $(1,1)$ and bounded on $L^p(\Heis_{d_1,d_1})$ for all $p \in \leftopenint 1,\infty\rightopenint$ whenever $\|F\|_{MW_2^s} < \infty$ for some $s > (\dim \Heis_{d_1,d_2} + (d_2-d_1)_+)/2$, where $\dim \Heis_{d_1,d_2}$ is the topological dimension $d_1d_2 + d_1 +d_2$ \cite[Corollary~6.1]{martini_joint}.

$\Space$ can be identified with the left quotient $\R^{d_1 \times d_2} \backslash \Heis_{d_1,d_2}$ via the projection map $(x,y,t) \mapsto (y,t+x^T y/2)$. Hence $\Heis_{d_1,d_2}$ acts by right translations on $\Space$, that is,
\[\tau_{(x,y,t)} : \Space \ni (z',z'') \mapsto (z'-y,z''-x^T z' -t + x^T y/2) \in \Space\]
is a measure-preserving affine transformation of $\Space$ for all $(x,y,t) \in \Heis_{d_1,d_2}$, and $\tau_{g h} = \tau_g \tau_h$. This in turn induces a unitary representation $\sigma$ of $\Heis_{d_1,d_2}$ on $L^2(\Space)$, given by $\sigma(g) f = f \circ \tau_g^{-1}$, and
\begin{equation}\label{eq:representation}
\begin{split}
T_{\jtwo} = d\sigma(-i\tilde T_{\jtwo}), &\qquad P_{\jone} T_{\jtwo} = d\sigma(-i\tilde X_{\jone,\jtwo}), \\
L_{\jone} = d\sigma\Biggl(-\tilde Y_{\jone}^2 - &\sum_{\jtwo} \tilde X_{\jone,\jtwo}^2 \Biggr), \qquad L = d\sigma(\tilde L).
\end{split}
\end{equation}
This shows in particular that the operators $L_1,\dots,L_{d_1},T_1,\dots,T_{d_2}$ (and all the polynomials in $L_1,\dots,L_{d_1},T_1,\dots,T_{d_2}$) are essentially self-adjoint on $C^\infty_c(\Space)$ and commute strongly (that is, their spectral resolutions commute), so they admit a joint functional calculus on $L^2(\Space)$ in the sense of the spectral theorem \cite[\S 3.1]{martini_spectral_2011}. Arguing analogously, by the use of the unitary representation $\varpi$ of $\R^{d_1} \times \R^{d_2}$ on $L^2(\Space)$ given by
\[(\varpi(u',u'') f)(x',x'') = e^{i \langle x' , u' \rangle} f(x', x''+u''),\]
one obtains that the operators $P_1,\dots,P_{d_1},T_1,\dots,T_{d_2}$ are essentially self-adjoint on $C^\infty_c(\Space)$ and commute strongly.

Because of the mentioned commutation properties, it is convenient to introduce in our notation the following ``vectors of operators'':
\[\vecL = (L_1,\dots,L_{d_1}), \qquad \vecT = (T_1,\dots,T_{d_2}), \qquad \vecP = (P_1,\dots,P_{d_1}).\]
Thus, for instance, $|\vecT|$ stands for the operator $(|T_1|^2 + \dots + |T_{d_2}|^2)^{1/2}$, that is, the square root $(-\Delta_{x''})^{1/2}$ of minus the second partial Laplacian on $\R^{d_1} \times \R^{d_2}$, while $|\vecP|$ is the operator of multiplication by $|x'|$. The subellipticity of $\tilde L$ then yields the following estimate.

\begin{proposition}
For all $\gamma \in \leftclosedint 0,\infty\rightopenint$ and $f \in L^2(\Space)$,
\begin{equation}\label{eq:fractionalpowers2}
\|\, |\vecP|^{\gamma} f\|_2 \leq C_\gamma \|L^{\gamma/2} |\vecT|^{-\gamma} f\|_2,
\end{equation}
where the $L^2$ norm on each side of \eqref{eq:fractionalpowers2} is understood to be $+\infty$ when $f$ does not belong to the domain of the corresponding operator.
\end{proposition}
\begin{proof}
We may assume $\gamma > 0$. Let $\vecP\vecT$ denote the double-indexed vector of operators $(P_{\jone} T_{\jtwo})_{\jone,\jtwo}$, and note that $|\vecP \vecT|^{\gamma} = |\vecP|^{\gamma} |\vecT|^{\gamma}$ (modulo closures). Moreover the spectrum $\leftclosedint 0,+\infty\rightopenint$ of $|\vecT|^{\gamma}$ is purely continuous, so $|\vecT|^{\gamma}$ is injective and its image is dense in $L^2(\Space)$. Therefore \eqref{eq:fractionalpowers2} is reduced to the proof of the inequality
\begin{equation}\label{eq:fractionalpowers}
\|\, |\vecP \vecT|^{\gamma} g\|_2 \leq C_\gamma \|L^{\gamma/2} g\|_2
\end{equation}
for all $g \in L^2(\Space)$.

By \eqref{eq:representation}, the differential operator $\tilde W = -\sum_{\jone,\jtwo} \tilde X_{\jone,\jtwo}^2$ on $\Heis_{d_1,d_2}$ corresponds to the operator $|\vecP \vecT|^2$ on $\Space$. Since $\tilde W$ is $\tilde\dil_r$-homogeneous, with the same homogeneity degree as the sublaplacian $\tilde L$, from \eqref{eq:representation} and \cite[Theorem 2.3(iv)]{martini_spectral_2011} we deduce \eqref{eq:fractionalpowers}
for all $\gamma \in 2\N$ and $g \in L^2(\Space)$.

We want now to extend \eqref{eq:fractionalpowers} to all the real $\gamma \geq 0$. For this, fix $m \in \N$ and let $A$ and $B$ be the closures of $|\vecP\vecT|^{2m}$ and $L^m$ on $L^2(\Space)$ respectively. Since $A$ and $B$ are nonnegative self-adjoint operators on $L^2(\Space)$, by \cite[Theorem 11.6.1]{martinezcarracedo_theory_2001}, for all $\theta \in \leftopenint 0,1 \rightopenint$,
\[(L^2(\Space),\dom A)_{[\theta]} = \dom A^\theta, \qquad (L^2(\Space),\dom B)_{[\theta]} = \dom B^\theta,\]
with equivalent norms, where $(\cdot,\cdot)_{[\theta]}$ denotes interpolation with respect to the complex method, and the domains of the various operators are endowed with the graph norms. On the other hand, \eqref{eq:fractionalpowers} implies that $\dom B \subseteq \dom A$, with continuous inclusion. By interpolation \cite[Theorem~4.1.2]{bergh_interpolation_1976}, we conclude that $\dom B^\theta \subseteq \dom A^\theta$, with continuous inclusion. This implies that
\[
\|\, |\vecP \vecT|^{\gamma} g\|_2 \leq C_\gamma (\|g\|_2 + \|L^{\gamma/2} g\|_2)
\]
for all $\gamma \in \leftclosedint 0,2m \rightclosedint$ and $g \in L^2(\Space)$. The bound \eqref{eq:fractionalpowers} then follows by replacing $f$ with $f \circ \dil_r$ in the previous inequality, exploiting the homogeneity relations \eqref{eq:homogeneity}, and taking the limit for $r \to \infty$.
\end{proof}

\section{Plancherel estimates}

From the previous section we know that the operators $L_1,\dots,L_{d_1},T_1,\dots,T_{d_2}$ have a joint functional calculus. In fact one can obtain a quite explicit formula for the integral kernel $\Kern_{G(\vecL,\vecT)}$ of an operator $G(\vecL,\vecT)$ in the functional calculus, in terms of the Hermite functions (cf.\ \cite[Proposition 3.1]{meyer_estimates_2006} for the case $d_1=d_2=1$, and \cite{strichartz_harmonic_1991} for the analogue on the Heisenberg groups). Namely, for all $\ell \in \N$, let $h_\ell$ denote the $\ell$-th Hermite function, that is,
\[
h_\ell(t) = (-1)^\ell (\ell! \, 2^\ell \sqrt{\pi})^{-1/2} e^{t^2/2} \left(\frac{d}{dt}\right)^\ell e^{-t^2},
\]
and set, for all $n \in \N^{d_1}$, $u \in \R^{d_1}$, $\xi \in \R^{d_2}$,
\[\tilde h_n(u,\xi) = |\xi|^{d_1/4} h_{n_1}(|\xi|^{1/2} u_1) \cdots h_{n_{d_1}}(|\xi|^{1/2} u_{d_1}).\]

\begin{proposition}
For all bounded Borel functions $G : \R^{d_1} \times \R^{d_2} \to \C$ compactly supported in $\R^{d_1} \times (\R^{d_2} \setminus \{0\})$,
\begin{multline*}
\Kern_{G(\vecL,\vecT)}(x,y) \\
= (2\pi)^{-d_2} \int_{\R^{d_2}} \sum_{n \in \N^{d_1}} G(|\xi|(2n+\vecONE), \xi) \, \tilde h_n(y',\xi) \, \tilde h_n(x',\xi) \, e^{i \langle \xi , x''-y'' \rangle} \,d\xi
\end{multline*}
for almost all $x,y \in \Space$, where $\vecONE = (1,\dots,1) \in \N^{d_1}$. In particular
\begin{equation}\label{eq:jointplancherel}
\|\Kern_{G(\vecL,\vecT)}(\cdot,y)\|_2^2 = (2\pi)^{-d_2} \int_{\R^{d_2}} \sum_{n \in \N^{d_1}} |G(|\xi|(2n+\vecONE),\xi)|^2 \, \tilde h_n^2(y',\xi) \,d\xi
\end{equation}
for almost all $y \in \Space$.
\end{proposition}
\begin{proof}
Let $\mathcal{F} : L^2(\R^{d_1} \times \R^{d_2}) \to L^2(\R^{d_1} \times \R^{d_2})$ be the isometry defined by
\[
\mathcal{F} \phi(x',\xi) = (2\pi)^{-d_2/2} \int_{\R^{d_2}} \phi(x',x'') \, e^{-i \langle \xi ,  x'' \rangle} \, dx'',
\]
i.e., the Fourier transform with respect to $x''$. Then
\begin{gather*}
\mathcal{F} L_{\jone}\phi (x',\xi) = L_{\jone,\xi} \, \mathcal{F} \phi(x',\xi), \qquad \mathcal{F} T_{\jtwo}\phi (x',\xi) = \xi_{\jtwo} \, \mathcal{F} \phi(x',\xi),
\end{gather*}
at least for $\phi$ in the Schwartz class, where
\[L_{\jone,\xi} =  (-i\partial_{x'_{\jone}})^2 +  |\xi|^2 (x'_{\jone})^2.\]
For all $\xi \neq 0$, $\{\tilde h_n(\cdot,\xi)\}_{n \in \N^{d_1}}$ is a complete orthonormal system for $L^2(\R^{d_1})$ made of real-valued functions and
\[L_{\jone,\xi} \, \tilde h_n(x',\xi) = (2n_{\jone}+1) |\xi| \, \tilde h_n(x',\xi).\]
In particular, if $\mathcal{G} : L^2(\R^{d_1} \times \R^{d_2}) \to L^2(\N^{d_1} \times \R^{d_2})$ is the isometry defined by
\[\mathcal{G} \psi(n,\xi) = \int_{\R^{d_1}} \psi(x',\xi) \, \tilde h_n(x',\xi) \,dx',\]
then
\[
\mathcal{G} \mathcal{F} L_{\jone} \phi(n,\xi) = (2n_{\jone}+1) |\xi| \, \mathcal{G} \mathcal{F} \phi(n,\xi), \qquad \mathcal{G} \mathcal{F} T_{\jtwo} \phi(n,\xi) = \xi_{\jtwo} \, \mathcal{G} \mathcal{F} \phi(n,\xi).
\]
The isometry $\mathcal{G}\mathcal{F}$ intertwines the operators $L_{\jone}$ and $T_{\jtwo}$ with some multiplication operators on $\N^{d_1} \times \R^{d_2}$, hence it intertwines the corresponding functional calculi:
\[
\mathcal{G} \mathcal{F} \, G(\vecL,\vecT) \, \phi(n,\xi) = G(|\xi|(2n+\vecONE),\xi) \, \mathcal{G} \mathcal{F} \phi(n,\xi).
\]
The inversion formulae for $\mathcal{F}$ and $\mathcal{G}$ and some easy manipulations then give the above expression for $\Kern_{G(\vecL,\vecT)}$. Moreover, if we set
\[
G_{y}(n,\xi) = (2\pi)^{-d_2/2} G(|\xi|(2n+\vecONE),\xi) \, \tilde h_n(y',\xi) \, e^{-i \langle \xi ,  y'' \rangle},
\]
then the formula for $\Kern_{G(\vecL,\vecT)}$ can be rewritten as
\[
\Kern_{G(\vecL,\vecT)}(\cdot,y) = (\mathcal{G}\mathcal{F})^{-1} G_{y},
\]
and since $\mathcal{G}\mathcal{F} : L^2(\Space) \to L^2(\N^{d_1} \times \R^{d_2})$ is an isometry we obtain \eqref{eq:jointplancherel}.
\end{proof}

If we restrict to the joint functional calculus of $L,T_1,\dots,T_{d_2}$, the formula \eqref{eq:jointplancherel} can be rewritten as follows. For all positive integers $d$, set $\N_{d} = 2\N + d$, and define, for all $N \in \N_{d}$ and $u \in \R^{d}$,
\[H_{d,N}(u) = \sum_{\substack{n_1,\dots,n_{d} \in \N \\ 2n_1 + \dots + 2 n_{d} + d = N}} h^2_{n_1}(u_1) \cdots h^2_{n_{d}}(u_{d}).\]

\begin{coro}\label{cor:plancherel}
For all bounded Borel functions $G : \R \times \R^{d_2} \to \C$ compactly supported in $\R \times (\R^{d_2} \setminus \{0\})$,
\begin{equation}\label{eq:plancherel}
\|\Kern_{G(L,\vecT)}(\cdot,y)\|_2^2 = (2\pi)^{-d_2} \int_{\R^{d_2}} \sum_{N \in \N_{d_1}} |G(N|\xi|,\xi)|^2 \, |\xi|^{d_1/2} H_{d_1,N}(|\xi|^{1/2} y') \,d\xi
\end{equation}
for almost all $y \in \Space$.
\end{coro}

We can now combine \eqref{eq:fractionalpowers2} and \eqref{eq:plancherel} to get the following weighted inequalities.

\begin{proposition}
For all $\gamma \geq 0$ and for all compactly supported bounded Borel functions $F : \R \to \C$,
\begin{equation}\label{eq:roughweightedplancherel}
\|\,|\vecP|^\gamma \Kern_{F(L)}(\cdot,y)\|_2^2 \leq C_\gamma \int_0^\infty |F(\lambda)|^2 \sum_{N \in \N_{d_1}} \frac{\lambda^{Q/2-\gamma}}{N^{Q/2-2\gamma}} \, 
H_{d_1,N}\left(\frac{\lambda^{1/2} y'}{N^{1/2}}\right) \,\frac{d\lambda}{\lambda}
\end{equation}
for almost all $y \in \Space$.
\end{proposition}
\begin{proof}
Let $G : \R \times \R^{d_2} \to \C$ be as in Corollary \ref{cor:plancherel}. In particular $\Kern_{G(L,\vecT)}(\cdot,y) \in L^2(\Space)$ for almost all $y \in \Space$, and from \cite[Theorem III.6.20]{dunford_linear_1958} and the definition of integral kernel one may deduce
\[L^{\gamma/2} |\vecT|^{-\gamma} \left( \Kern_{G(L,\vecT)}(\cdot, y) \right) = \Kern_{L^{\gamma/2} |\vecT|^{-\gamma} G(L,\vecT)}(\cdot, y)\]
for all $\gamma \geq 0$ and almost all $y \in \Space$. This equality, together with \eqref{eq:plancherel} and \eqref{eq:fractionalpowers2}, implies that
\[
\|\,|\vecP|^\gamma \Kern_{G(L,\vecT)}(\cdot,y)\|_2^2 \leq C_\gamma \sum_{N \in \N_{d_1}} \int_{\R^{d_2}} |G(N|\xi|,\xi)|^2 \, N^\gamma \, |\xi|^{d_1/2-\gamma} \, H_{d_1,N}(|\xi|^{1/2} y') \,d\xi.
\]

Choose now an increasing sequence $(\zeta_n)_{n \in \N}$ of nonnegative Borel functions on $\R$, compactly supported in $\R \setminus \{0\}$ and converging pointwise on $\R \setminus \{0\}$ to the constant $1$, and define $G_n(\lambda,\xi) = F(\lambda) \,\zeta_n(|\xi|)$. Note that $\Kern_{F(L)}(\cdot,y) \in L^2(\Space)$ for almost all $y \in \Space$ \cite[Lemma 2.2]{duong_plancherel-type_2002}, hence
\[\Kern_{G_n(L,\vecT)}(\cdot,y) = \zeta_n(|\vecT|) (\Kern_{F(L)}(\cdot,y))\]
for almost all $y \in \Space$, as before, and $\Kern_{G_n(L,\vecT)}(\cdot,y) \to \Kern_{F(L)}(\cdot,y)$ in $L^2(\Space)$ for almost all $y$, because $|\vecT|$ has trivial kernel. The conclusion then follows by applying the previous inequality when $G = G_n$ and letting $n$ tend to infinity.
\end{proof}

Now we recall some well-known estimates for the Hermite functions which we need in the sequel.

\begin{lemma}
For all $N = 2n+1 \in \N_1$,
\begin{equation}\label{eq:muckenhoupt}
H_{1,N}(u) = h_n^2(u) \leq \begin{cases}
C(N^{1/3} + |u^2-N|)^{-1/2} &\text{for all $u \in \R$.}\\
C\exp(-cu^2) &\text{when $u^2 \geq 2N$,}
\end{cases}
\end{equation}
Moreover, if $d \geq 2$, then, for all $N \in \N_d$,
\begin{equation}\label{eq:higherbounds}
H_{d,N}(u) \leq \begin{cases}
CN^{d/2-1} &\text{for all $u \in \R^d$,}\\
C\exp(-c|u|_\infty^2) &\text{when $|u|_\infty^2 \geq 2N$,}
\end{cases}
\end{equation}
where $|u|_\infty = \max\{|u_1|,\dots,|u_d|\}$.
\end{lemma}
\begin{proof}
For the bounds \eqref{eq:muckenhoupt}, see \cite[(2.3), p.\ 435]{muckenhoupt_mean_1970} or \cite[Lemma 1.5.1]{thangavelu_lectures_1993}. For the first inequality in \eqref{eq:higherbounds}, see \cite[Lemma~3.2.2]{thangavelu_lectures_1993}; the second inequality is an easy consequence of \eqref{eq:muckenhoupt}.
\end{proof}

These bounds allow us to obtain the following crucial estimate.

\begin{lemma}\label{lem:higherhermite}
For all fixed $d \in \N \setminus \{0\}$ and $\epsilon \in \leftopenint 0,\infty \rightopenint$, the sum
\begin{equation}\label{eq:higherhermitesum}
\sum_{N \in \N_d} \frac{\max\{1,|u|\}^{\epsilon}}{N^{d/2+\epsilon}} \, H_{d,N} \left( \frac{u}{N^{1/2}}\right) 
\end{equation}
has a finite upper bound, independent of $u \in \R^d$.
\end{lemma}
\begin{proof}
We split the sum into several parts, and use the bounds \eqref{eq:muckenhoupt}, \eqref{eq:higherbounds}.

The part where $N \leq |u|/2$ is empty unless $|u| \geq 1$; in this case, moreover, $|N^{-1/2} u|^2 \geq 4N$, hence $H_{d,N}(N^{-1/2} u) \leq C \exp(-c|u|^2/N)$, and
\[\sup_u \sum_{N \leq |u|/2} |u|^\epsilon N^{-d/2-\epsilon} \exp(-c|u|^2/N) \leq \sum_{N \in \N_d} \sup_{t \geq 4N} t^{\epsilon/2} \exp(-ct),\]
which is finite. 

If $N \geq |u|/2$ and $d \geq 2$, then $H_{d,N}(N^{-1/2} u) \leq C N^{d/2-1}$, and
\[\sup_u \sum_{N \geq |u|/2} \max\{1,|u|\}^\epsilon N^{-1-\epsilon} < \infty.\]
When $d = 1$, the same argument works for the part where $N \geq 2|u|$, because in this case $|N^{-1/2} u|^2 \leq N/4$ and the bound $H_{N}(N^{-1/2} u) \leq C N^{-1/2}$ holds. However the part where $|u|/2 < N < 2|u|$ requires a different estimate.

Namely, the part of \eqref{eq:higherhermitesum} where $|u|/2 < N \leq |u|-1$ is majorized by
\[C_\epsilon \, |u|^{-1} \sum_{|u|/2 < N < |u|-1} |1-N/u|^{-1/2} \leq C_\epsilon \int_{1/2}^1 |1-t|^{-1/2} \,dt,\]
which is finite and independent of $u$. Analogously one bounds the part of \eqref{eq:higherhermitesum} where $|u|+1 \leq N < 2|u|$. The remaining part, where $|u|-1 < N < |u|+1$, contains at most one summand, which moreover is bounded by a constant.
\end{proof}

The previous inequality allows us to simplify \eqref{eq:roughweightedplancherel} considerably and to obtain the weighted Plancherel estimates announced in the introduction. Recall that $w_R$ denotes the weight function defined by \eqref{eq:weight} for all $R > 0$.

\begin{proposition}\label{prp:weightedplancherel}
For all $\gamma \in \leftclosedint 0,d_2/2 \rightopenint$ and all bounded compactly supported Borel functions $F : \R \to \C$,
\[
\|\, |\vecP|^\gamma \Kern_{F(L)}(\cdot,y)\|_2^2 \leq C_\gamma \int_0^\infty |F(\lambda)|^2 \, \lambda^{(d_1+d_2)/2} \, \min\{\lambda^{d_2/2-\gamma},|y'|^{2\gamma-d_2}\} \,\frac{d\lambda}{\lambda}
\]
for almost all $y \in \Space$. In particular, for all $R > 0$, if $\supp F \subseteq \leftclosedint R^2,4R^2 \rightclosedint$, then
\[
\esssup_{y \in \Space}  |B(y,R^{-1})|^{1/2} \, \| (1+w_R(\cdot,y))^\gamma \Kern_{F(L)}(\cdot,y)\|_2 \leq C_{\gamma} \|F_{(R^2)}\|_{L^2},
\]
where the constant $C_{\gamma}$ does not depend on $R$.
\end{proposition}
\begin{proof}
In view of \eqref{eq:roughweightedplancherel}, the first inequality follows immediately from Lemma~\ref{lem:higherhermite} with $d = d_1$ and $\epsilon = d_2-2\gamma$. In the case $\supp F \subseteq \leftclosedint R^2,4R^2 \rightclosedint$, a simple manipulation, together with \eqref{eq:weight} and \eqref{eq:ballmeasure}, gives the second inequality.
\end{proof}

\section{The multiplier theorems}

We now show how the weighted Plancherel estimates obtained in the previous section can be used to improve the known multiplier theorems for the Grushin operator. First we recall the basic known estimates for operators satisfying Gaussian-type heat kernel bounds in a doubling metric-measure space.

\begin{proposition}
For all $R > 0$, $\alpha \geq 0$, $\beta > \alpha$, and for all functions $F : \R \to \C$ such that $\supp F \subseteq [-4R^2,4R^2]$,
\begin{equation}\label{eq:standardl2weighted}
\esssup_{y \in \Space} |B(y,R^{-1})|^{1/2} \|(1+ R\dist(\cdot,y))^\alpha \Kern_F(\cdot,y)\|_2 \leq C_{\alpha,\beta} \|F_{(R^2)}\|_{W_\infty^\beta},
\end{equation}
where the constant $C_{\alpha,\beta}$ does not depend on $R$. If in addition  $\beta > \alpha+Q/2$, then
\begin{equation}\label{eq:standardl1weighted}
\esssup_{y \in \Space} \|(1+ R\dist(\cdot,y))^\alpha \Kern_F(\cdot,y)\|_1 \leq C_{\alpha,\beta} 
\|F_{(R^2)} \|_{W_\infty^\beta},
\end{equation}
where again $C_{\alpha,\beta}$ does not depend on $R$.
\end{proposition}
\begin{proof}
For the first inequality, see \cite{hebisch_functional_1995} or \cite[Lemma 4.3]{duong_plancherel-type_2002}; note that the statement in \cite{duong_plancherel-type_2002} seems to require that the multiplier $F$ is supported away from the origin, but its proof clarifies that this is not necessary, because here we do not perform the change of variable $\lambda \mapsto \sqrt{\lambda}$ in the multiplier function. The second inequality is an immediate consequence of the first, via H\"older's inequality and \cite[Lemma 4.4]{duong_plancherel-type_2002}.
\end{proof}

These inequalities can be improved by means of the weighted Plancherel estimates. For this, some properties of the weight functions $w_R$ are needed.

\begin{lemma}
Suppose that $0 \leq \gamma < \min\{d_1,d_2\}/2$ and $\beta > Q/2 - \alpha$. For all $y \in \Space$ and $R > 0$,
\begin{equation}\label{eq:volumebound}
\int_\Space (1+w_R(x,y))^{-2\gamma} (1 + R \dist(x,y))^{-2\beta} \,dx \leq C_{\alpha,\beta} |B(y,R^{-1})|.
\end{equation}
Moreover, for all $x,y \in \Space$ and $R > 0$,
\begin{equation}\label{eq:distancebound}
w_R(x,y) \leq C (1 + R \dist(x,y)).
\end{equation}
\end{lemma}
\begin{proof}
By exploiting the homogeneity properties of the distance $\dist$ and the weights $w_R$, we may suppose that $R = 1$. Then \eqref{eq:distancebound} immediately follows from the fact that
\[\min\{1,|y'|^{-1}\}|x'| \leq 1 + |x' - y'| \leq C (1 + \dist(x,y)),\]
by \eqref{eq:controldistance}.

To show  \eqref{eq:volumebound} we note that  by translation-invariance we may also suppose that $y'' = 0$. By \eqref{eq:ballmeasure}, we must then prove that
\[
\int_\Space \left(1+\frac{|x'-y'|}{1+|y'|}\right)^{-2\gamma} (1 + \dist(x,y))^{-2\beta} \,dx \leq C_{\gamma,\beta} (1 + |y'|)^{d_2}.
\]
We split the integral into two parts, according to the asymptotics \eqref{eq:controldistance}. In the region $\Space_1 = \{x \in \Space \tc |x''|^{1/2} \geq |x'| + |y'|\}$, we decompose $\beta = \beta_1 + \beta_2$ so that $\beta_1 > d_1/2-\gamma$ and $\beta_2 > d_2$, whence the integral on $\Space_1$ is at most
\[
(1+|y'|)^{2\gamma} \int_{\R^{d_1}} (1+|x'-y'|)^{-2(\gamma+\beta_1)} \,dx' \int_{\R^{d_2}} (1+|x''|^{1/2})^{-2\beta_2} \,dx''.
\]
In the region $\Space_2 = \{x \in \Space \tc |x''|^{1/2} < |x'| + |y'|\}$, instead, we decompose $\beta = \tilde\beta_1+\tilde\beta_2$ so that $\tilde\beta_1 > (d_1+d_2)/2-\gamma$ and $\tilde\beta_2 > d_2/2$, whence the integral on $\Space_2$ is at most
\[\begin{split}
\int_\Space &\left(1+\frac{|x'-y'|}{1+|y'|}\right)^{-2\gamma} (1+|x'-y'|)^{-2\tilde\beta_1} \left(1 + \frac{|x''|}{|x'|+|y'|}\right)^{-2\tilde\beta_2} \,dx \\
&\leq C_{\gamma,\beta} \int_{\R^{d_1}} \left(1+\frac{|u|}{1+|y'|}\right)^{-2\gamma} (1+|u|)^{-2\tilde\beta_1} (|u+y'|+|y'|)^{d_2} \,du \\
&\leq C_{\gamma,\beta} \left( (1+|y'|)^{2\gamma} \int_{\R^{d_1}} (1+|u|)^{-2\nu} \,du + |y'|^{d_2} \int_{\R^{d_1}} (1+|u|)^{-2\tilde\beta_1} \,du \right),
\end{split}\]
where $\nu = \tilde\beta_1 + \gamma - d_2/2 > d_1/2$. The conclusion follows.
\end{proof}

A strengthened weighted version of \eqref{eq:standardl2weighted} can now be obtained using the Mauceri-Meda interpolation trick \cite{mauceri_vectorvalued_1990} (see also \cite[\S 3]{martini_joint} and \cite[Lemma 4.3]{duong_plancherel-type_2002}).

\begin{proposition}\label{prp:speciall2weighted}
For all $R > 0$, $\alpha \geq 0$, $\beta > \alpha$, $\gamma \in \leftclosedint 0,d_2/2\rightopenint$, and for all functions $F : \R \to \C$ such that $\supp F \subseteq \leftclosedint R^2,4R^2 \rightclosedint$,
\begin{multline}\label{eq:speciall2weighted}
\esssup_{y \in \Space}  |B(y,R^{-1})|^{1/2} \, \| (1 + R\dist(\cdot,y))^\alpha (1+w_R(\cdot,y))^\gamma \Kern_{F(L)}(\cdot,y)\|_2 \\
\leq C_{\alpha,\beta,\gamma} \|F_{(R^2)} \|_{W_2^\beta},
\end{multline}
where the constant $C_{\alpha,\beta,\gamma}$ does not depend on $R$.
\end{proposition}
\begin{proof}
The estimate \eqref{eq:standardl2weighted}, together with \eqref{eq:distancebound} and a Sobolev embedding, immediately implies \eqref{eq:speciall2weighted} in the case $\beta > \alpha + d_2/2 + 1/2$. On the other hand, in the case $\alpha = 0$, \eqref{eq:speciall2weighted} is given by Proposition~\ref{prp:weightedplancherel} for all $\beta > 0$. The conclusion then follows by interpolation (see, e.g., \cite{bergh_interpolation_1976,cwikel_interpolation_1984}).
\end{proof}

An alternative proof of Proposition~\ref{prp:speciall2weighted} can be obtained using minor adjustments of the technique developed in \cite{cowling_spectral_2001}.

Let $D = Q - \min\{d_1,d_2\} = \max\{d_1+d_2,2d_2\}$. Proposition~\ref{prp:speciall2weighted}, together with \eqref{eq:volumebound} and H\"older's inequality, then yields an improvement of \eqref{eq:standardl1weighted}.

\begin{coro}
For all $R > 0$, $\alpha \geq 0$, $\beta > \alpha + D/2$, and for all functions $F : \R \to \C$ such that $\supp F \subseteq \leftclosedint R^2,4R^2 \rightclosedint$,
\begin{equation}\label{eq:speciall1weighted}
\esssup_{y \in \Space} \| (1 + R\dist(\cdot,y))^\alpha \Kern_{F(L)}(\cdot,y)\|_1 
\leq C_{\alpha,\beta} \|F_{(R^2)}\|_{W_2^\beta},
\end{equation}
where the constant $C_{\alpha,\beta}$ does not depend on $R$. In particular, under the same hypotheses,
\begin{equation}\label{eq:l1offball}
\esssup_{y \in \Space} \int_{\Space \setminus B(y,r)} |\Kern_{F(L)}(x,y)| \,dx 
\leq C_{\alpha,\beta} (1+rR)^{-\alpha} \|F_{(R^2)} \|_{W_2^\beta}.
\end{equation}
\end{coro}

We are finally able to prove our main results.

\begin{proof}[Proof of Theorem~\ref{thm:maintheorem}]
We can follow the lines of the proof of \cite[Theorem 3.1]{duong_plancherel-type_2002}, where the inequality (4.18) there is replaced by our \eqref{eq:l1offball}.
\end{proof}

\begin{proof}[Proof of Theorem~\ref{thm:maintheoremriesz}]
Choose $\beta \in \leftopenint D/2, \kappa+1/2 \rightopenint$. Let $\eta \in C^\infty_c(\R)$ be supported in $\leftclosedint -1/2,1/2\rightclosedint$ and equal to $1$ in a neighborhood of the origin, and set $F(\lambda) = (1-|\lambda|)_+^\kappa$. The function $\eta F$ is smooth and compactly supported, while $(1-\eta)F$ is compactly supported away from the origin and belongs to $W_2^\beta$. The inequalities \eqref{eq:standardl1weighted} and \eqref{eq:speciall1weighted} then imply that the operators $\eta(tL) F(tL)$ and $(1-\eta(tL)) F(tL)$ are bounded on $L^1(\Space)$, uniformly in $t > 0$, and the same holds for their sum $(1-tL)_+^\kappa$. The conclusion follows by self-adjointness and interpolation.
\end{proof}

\section{Sharpness of the obtained results}\label{section:sharp}

The aim of this section is to show that, if $d_1 \ge d_2$, then the result in Theorem~\ref{thm:maintheorem} is sharp. More precisely, if $d_1 \ge d_2$ and $s<D/2=(d_1+d_2)/2$, then the first inequality in \eqref{eq:normbounds} cannot hold. Indeed, if we consider the functions 
$H_t(\lambda)=\lambda^{it}$, then, for $t>1$,
\[
C \|H_t \|_{MW_2^s} \sim t^s.
\]
On the other hand, we make the following observation. 
\begin{proposition}
Suppose that $L$ is the Grushin operator acting on $\Space = \R^{d_1}\times \R^{d_2}$.
Then the following lower bounds holds:
\[
\|H_t(L)\|_{L^1\to L^{1,w}}=\|L^{it}\|_{L^1\to L^{1,w}} \ge C (1+|t|)^{(d_1+d_2)/2}
\]
for all $t>0$. 
\end{proposition}
\begin{proof}
Because the Grushin operator is elliptic on $\Space_0 = \{ x \in \Space \tc x' \neq 0\}$, one can use the same argument as in \cite{sikora_imaginary_2001} to prove that, for all $y \in \Space_0$,
\[
|p_t(x,y) - |y'|^{- d_2}(4\pi
t)^{-(d_1+d_2)/2}{\rm e}^{-\dist(x,y)^2/4t}|
 \le C
t^{1/2}t^{-(d_1+d_2)/2}
\]
for all $x$ in a small neighborhood of $y$ and all $t \in \leftopenint 0,1 \rightopenint$. Here $p_t=\Kern_{\exp(-tL)}$ is the heat kernel corresponding to the Grushin operator. The rest of the argument is the same as in \cite{sikora_imaginary_2001}, so we skip it here. 
\end{proof}

\bibliographystyle{abbrv}

\begin{thebibliography}{10}

\bibitem{arcozzi_grushin_2008}
N.~Arcozzi and A.~Baldi.
\newblock From {G}rushin to {H}eisenberg via an isoperimetric problem.
\newblock {\em J. Math. Anal. Appl.}, 340(1):165--174, 2008.

\bibitem{beals_transversally_2004}
R.~Beals, B.~Gaveau, P.~Greiner, and Y.~Kannai.
\newblock Transversally elliptic operators.
\newblock {\em Bull. Sci. Math.}, 128(7):531--576, 2004.

\bibitem{bergh_interpolation_1976}
J.~Bergh and J.~L{\"o}fstr{\"o}m.
\newblock {\em Interpolation spaces. {A}n introduction}.
\newblock Springer-Verlag, Berlin, 1976.
\newblock Grundlehren der Mathematischen Wissenschaften, No. 223.

\bibitem{cowling_spectral_2001}
M.~Cowling and A.~Sikora.
\newblock A spectral multiplier theorem for a sublaplacian on {$\rm SU(2)$}.
\newblock {\em Math. Z.}, 238(1):1--36, 2001.

\bibitem{cwikel_interpolation_1984}
M.~Cwikel and S.~Janson.
\newblock Interpolation of analytic families of operators.
\newblock {\em Studia Math.}, 79(1):61--71, 1984.

\bibitem{dunford_linear_1958}
N.~Dunford and J.~T. Schwartz.
\newblock {\em Linear {O}perators. {I}. {G}eneral {T}heory}.
\newblock With the assistance of W. G. Bade and R. G. Bartle. Pure and Applied
  Mathematics, Vol. 7. Interscience Publishers, Inc., New York, 1958.

\bibitem{duong_singular_1999}
X.~T. Duong and A.~McIntosh.
\newblock Singular integral operators with non-smooth kernels on irregular
  domains.
\newblock {\em Rev. Mat. Iberoamericana}, 15(2):233--265, 1999.

\bibitem{duong_plancherel-type_2002}
X.~T. Duong, E.~M. Ouhabaz, and A.~Sikora.
\newblock Plancherel-type estimates and sharp spectral multipliers.
\newblock {\em J. Funct. Anal.}, 196(2):443--485, 2002.

\bibitem{hebisch_multiplier_1993}
W.~Hebisch.
\newblock Multiplier theorem on generalized {H}eisenberg groups.
\newblock {\em Colloq. Math.}, 65(2):231--239, 1993.

\bibitem{hebisch_functional_1995}
W.~Hebisch.
\newblock Functional calculus for slowly decaying kernels.
\newblock 1995.
\newblock Preprint. Available on the web at
  \texttt{http://www.math.uni.wroc.pl/\textasciitilde{}hebisch/}.

\bibitem{jerison_subelliptic_1987}
D.~Jerison and A.~S{\'a}nchez-Calle.
\newblock Subelliptic, second order differential operators.
\newblock In {\em Complex analysis, {III} ({C}ollege {P}ark, {M}d., 1985--86)},
  volume 1277 of {\em Lecture Notes in Math.}, pages 46--77. Springer, Berlin,
  1987.

\bibitem{jotsaroop_riesz_2011}
K.~Jotsaroop, P.~K. Sanjay, and S.~Thangavelu.
\newblock {R}iesz transforms and multipliers for the {G}rushin operator.
\newblock 2011.
\newblock To appear in \textit{J.\ Analyse Math.} \texttt{arXiv:1110.3227}.

\bibitem{martinezcarracedo_theory_2001}
C.~Mart{\'{\i}}nez~Carracedo and M.~Sanz~Alix.
\newblock {\em The theory of fractional powers of operators}, volume 187 of
  {\em North-Holland Mathematics Studies}.
\newblock North-Holland Publishing Co., Amsterdam, 2001.

\bibitem{martini_multipliers_2010}
A.~Martini.
\newblock {\em {Algebras of differential operators on Lie groups and spectral
  multipliers}}.
\newblock {Tesi di perfezionamento (PhD thesis)}, Scuola Normale Superiore,
  Pisa, 2010.
\newblock \texttt{arXiv:1007.1119}.

\bibitem{martini_joint}
A.~Martini.
\newblock {Analysis of joint spectral multipliers on Lie groups of polynomial
  growth}.
\newblock 2010.
\newblock To appear in \textit{Ann. Inst. Fourier (Grenoble)}.
  \texttt{arXiv:1010.1186}.

\bibitem{martini_spectral_2011}
A.~Martini.
\newblock {Spectral theory for commutative algebras of differential operators
  on Lie groups}.
\newblock {\em J. Funct. Anal.}, 260(9):2767--2814, 2011.

\bibitem{mauceri_vectorvalued_1990}
G.~Mauceri and S.~Meda.
\newblock Vector-valued multipliers on stratified groups.
\newblock {\em Rev. Mat. Iberoamericana}, 6(3-4):141--154, 1990.

\bibitem{meyer_estimates_2006}
R.~Meyer.
\newblock {\em {$L^p$}-estimates for the wave equation associated to the
  {G}ru\v sin operator}.
\newblock {PhD} dissertation, {Christian-Albrechts-Universit\"at} zu Kiel,
  2006.

\bibitem{muckenhoupt_mean_1970}
B.~Muckenhoupt.
\newblock Mean convergence of {H}ermite and {L}aguerre series. {I}, {II}.
\newblock {\em Trans. Amer. Math. Soc. 147 (1970), 419-431; ibid.},
  147:433--460, 1970.

\bibitem{mller_spectral_1994}
D.~M{\"u}ller and E.~M. Stein.
\newblock On spectral multipliers for {H}eisenberg and related groups.
\newblock {\em J. Math. Pures Appl. (9)}, 73(4):413--440, 1994.

\bibitem{robinson_analysis_2008}
D.~W. Robinson and A.~Sikora.
\newblock Analysis of degenerate elliptic operators of {G}ru\v sin type.
\newblock {\em Math. Z.}, 260(3):475--508, 2008.

\bibitem{sikora_imaginary_2001}
A.~Sikora and J.~Wright.
\newblock Imaginary powers of {L}aplace operators.
\newblock {\em Proc. Amer. Math. Soc.}, 129(6):1745--1754, 2001.

\bibitem{strichartz_harmonic_1991}
R.~S. Strichartz.
\newblock {$L^p$} harmonic analysis and {R}adon transforms on the {H}eisenberg
  group.
\newblock {\em J. Funct. Anal.}, 96(2):350--406, 1991.

\bibitem{thangavelu_lectures_1993}
S.~Thangavelu.
\newblock {\em Lectures on {H}ermite and {L}aguerre expansions}, volume~42 of
  {\em Mathematical Notes}.
\newblock Princeton University Press, Princeton, NJ, 1993.
\newblock With a preface by Robert S. Strichartz.

\bibitem{torreslopera_cohomology_1985}
J.~F. Torres~Lopera.
\newblock The cohomology and geometry of {H}eisenberg-{R}eiter nilmanifolds.
\newblock In {\em Differential geometry, {P}e\~n\'\i scola 1985}, volume 1209
  of {\em Lecture Notes in Math.}, pages 292--301. Springer, Berlin, 1986.

\end{thebibliography}

\end{document}